\documentclass[twoside,12pt]{article}
\usepackage{hyperref}
\usepackage{times,amsmath,amssymb,amsthm,cite}
\usepackage{graphicx}
\usepackage{float}
\usepackage{epstopdf} 
\usepackage{amssymb}
\usepackage{amsmath}
\usepackage{mathtools}
\numberwithin{equation}{section}
\numberwithin{figure}{section}
\newtheorem{theorem}{Theorem}[section]
\newtheorem{lemma}[theorem]{Lemma}
\newtheorem{corollary}[theorem]{Corollary}
\newtheorem{proposition}[theorem]{Proposition}
\theoremstyle{definition}
\newtheorem{definition}[theorem]{Definition}
\newtheorem{example}[theorem]{Example}
\theoremstyle{remark}

\everymath{\displaystyle}
\date{}

\begin{document}


\begin{center}
{\bf Locally Product-like Statistical Manifolds and Their Hypersurfaces}

\bigskip

Esra Erkan, Kazuhiko Takano and    Mehmet G\"{u}lbahar
\end{center}

\bigskip

\begin{quote}
\noindent {\bf Abstract.}
In this paper, almost product-like Riemannian manifolds are investigated. Basic properties on tangential
hypersurfaces of almost product-like Riemannian manifolds are obtained. Some examples of tangential hypersurfaces are presented. Some relations involving the Riemannian curvature tensor of a tangential hypersurface are computed.

\end{quote}

{\rm \ }

\begin{quote}
\medskip

\noindent {\em AMS 2000 Mathematics Subject Classification\/}:
53C05, 53C40, 53C42.

\noindent {\em Keywords\/}: Riemann product manifold, hypersurface, statistical manifold
\end{quote}

\section{Introduction}
The concept of statistical manifolds has important application areas and various properties of statistical structures in geometric and physical terms have been studied intensively recently.
In physics, there exist various application areas of statistical manifolds such as neural networks,
machine learning, artificial intelligence, black holes \cite{Calin,Efron,Gucht,Vos}.
 Statistical manifolds were firstly introduced by S. Amari \cite{Amari-1985-book} in 1985. Later, these structures drew attention of
several authors. Some basic properties of hypersurfaces of statistical manifolds were presented by H. Furuhata in \cite{Furuhata-2009,Furuhata-2011}, hypersurfaces of Sasakian statistical manifolds were investigated by F. Wu, Y. Jiang, and L. Zhang in \cite{Feng}, hypersurfaces of almost Hermitian statistical manifolds were investigated by H.  Akbari and F. Malek in \cite{Akbari}, etc.

As a new perspective on  Hermitian statistical manifolds, the second author was firstly introduced the notion of Hermite-like manifolds in \cite{Takano,Takano:1}.
The most important aspect of this point of view is that it reveals to us a more general form of Hermitian geometry, which has a long history and on which many mathematicians, physicists, and engineers have investigated. In fact, a Hermite-like manifold becomes a Hermitian manifold when two almost complex structures  $J$ and $J^{\ast}$ are equal.

Inspired by the definition of Hermite-like manifolds, contact-like manifolds were studied in \cite{Aytimur,Erken,Takano,Takano:1},
para-Kaehler-like and quaternionic-like Kaehler manifolds were investigated in \cite{Vilcu,Vilcu:2}, the Wintgen-like inequality on statistical warped product manifolds was
computed in \cite{Murathan}, etc.

The main purpose of this paper is to present the basic properties of almost product-like Riemannian manifolds
and their hypersurfaces. With these reviews and using some obtained properties for hypersurfaces of almost
product-like Riemannian manifolds, we will be able to the opportunity to give some basic definitions of para contact-like Riemannian manifolds.

\section{Preliminaries\label{sect-prel}}
Let $(\widetilde{M},\widetilde{g})$ be a Riemannian manifold.  Denote a torsion-free affine connection by $\widetilde{\nabla}$. The triple $(\widetilde{M},\widetilde{g},\widetilde{\nabla})$ is called a statistical manifold if $\widetilde{\nabla}\widetilde{g}$ is symmetric. For the statistical manifold, we define another affine connection $\widetilde{\nabla}^{\ast}$ by
\begin{equation}
Z\widetilde{g}(X,Y) =\widetilde{g}(X,\widetilde{\nabla}_{Z}^{\ast }Y)+\widetilde{g}(\widetilde{\nabla}_{Z}X,Y), \label{eq8}
\end{equation}
 for any vector fields $X, Y$ and $Z$ on $\widetilde{M}$. The affine connection $\widetilde{\nabla}^{\ast}$ is called dual of $\widetilde{\nabla}$ with respect to $\widetilde{g}$. The affine connection $\widetilde{\nabla}^{\ast}$ is torsion-free and satisfies
$(\widetilde{\nabla}^{\ast})^{\ast}=\widetilde{\nabla}$. It is clear to see that $\widetilde{\nabla}^{\circ}=\frac{1}{2}\left(\widetilde{\nabla}+\widetilde{\nabla}^{\ast }\right)$ is a metric connection. The pair $(\widetilde{g},\widetilde{\nabla})$ is a statistical structure on $\widetilde{M}$ is and only if so is $(\widetilde{g},\widetilde{\nabla}^{\ast})$. Therefore the triple $(\widetilde{M},\widetilde{g},\widetilde{\nabla}^{\ast})$ is a statistical manifold, too. We denote by $\widetilde{R}$ and $\widetilde{R}^{\ast}$ the curvature tensors on $\widetilde{M}$ with respect to the affine connection $\widetilde{\nabla}$ and its conjugate $\widetilde{\nabla}^{\ast}$, respectively. Then we find
\begin{equation*}
 \widetilde{g}(\widetilde{R}(X,Y)Z,W)=-\widetilde{g}(Z,\widetilde{R}^{\ast}(X,Y)W)
\end{equation*}
 where $\widetilde{R}(X, Y)Z=\widetilde{\nabla}_{X}(\widetilde{\nabla}_{Y}Z)-\widetilde{\nabla}_{Y}(\widetilde{\nabla}_{X}Z)-\widetilde{\nabla}_{[X,Y]}Z$. Therefore $\widetilde{R}$ vanishes identically if and only if so is $\widetilde{R}^{\ast}$. We call flat if $\widetilde{R}$  vanishes identically. If the curvature tensor $\widetilde{R}$  with respect to the affine connection $\widetilde{\nabla}$ satisfies
\begin{equation*}
\widetilde{R}(X,Y)Z=\widetilde{c}\left\{ \widetilde{g}(Y,Z)X-\widetilde{g}(X,Z)Y \right\},
\end{equation*}
then the statistical manifold $(\widetilde{M},\widetilde{g},\widetilde{\nabla})$ is called a space of constant curvature $\widetilde{c}$.

Suppose that $J$ and $J^{\ast}$ are  almost complex structures  of type $(1,1)$ satisfying the condition
\begin{eqnarray}
\widetilde{g}\left(JX,Y\right)=-\widetilde{g}\left(X,J^{\ast}Y\right) \label{com:1}
\end{eqnarray}
for any tangent vector fields $X,Y\in \Gamma(T\widetilde{M})$. Then $(\widetilde{M},\widetilde{g},J)$ is called an almost Hermite-like manifold.

Let $F$ be a tensor field of type $(1,1)$ on $\widetilde{M}$. If
\begin{equation} \label{eq1}
    F^{2}=I,
\end{equation}
where $I$ denotes the identity map, then the pair $(\widetilde{M}, F)$  is called an almost product manifold with
almost product structure $F$. For any almost product manifold, there exist two projections $P$ and $Q$ satisfying
\begin{eqnarray*}
F=P-Q.
\end{eqnarray*}
 If $F=I$ (resp. $F=-I$), we find $P=I$, $Q=0$ (resp. $P=0$, $Q=I$).  We discuss the case of $F\neq \pm I$ in this paper.
If an almost product manifold admits a Riemannian metric $\widetilde{g}$ such that the condition
\begin{eqnarray*}
\widetilde{g}(F X,Y)=\widetilde{g}(X,FY)
\end{eqnarray*}
or, equivalently,
\begin{eqnarray*}
\widetilde{g}(FX,FY)=\widetilde{g}(X,Y)
\end{eqnarray*}
is satisfied for any $X,Y\in \Gamma(T\widetilde{M})$, then $(\widetilde{M}, \widetilde{g},F)$ is called an almost product Riemannian manifold \cite{Yano}.

Let $(\widetilde{M},\widetilde{g},F)$ be an almost product Riemannian manifold and $\widetilde{\nabla}^{\circ}$ be the Levi-Civita connection of $(\widetilde{M}, \widetilde{g},F)$. In particular, if $\widetilde{\nabla}^{{\circ}}F=0$, then $(\widetilde{M}, \widetilde{g},F)$ is called a locally product Riemannian manifold \cite{Yano}.

For any hypersurface $(M,g)$ of $(\widetilde{M},\widetilde{g},F)$, the Gauss and Weingarten formulas are given as follows:
\begin{equation}
\widetilde{\nabla}^{{\circ}}_{X}Y=\nabla^{{\circ}}_{X}Y+g(A^{{\circ}}_{N}X,Y)N,
\end{equation}
\begin{equation}
\widetilde{\nabla}^{{\circ}}_{X}N=-A^{{\circ}}_{N}X \label{Wein:o}
\end{equation}
for any $X,Y\in \Gamma(TM)$, where $N$ is the unit vector field, $\nabla^{{\circ}}$ is the induced linear connection and $A^{{\circ}}_N$ is the shape operator on $(M,g)$.

The hypersurface is called totally geodesic if $A^{{\circ}}_N$ vanishes identically, and it is called  totally umbilical if there exists a smooth function $\rho^{\circ}$ on $M$ such that
\begin{equation} \label{eq4}
    A^{{\circ}}_NX=\rho^{\circ} X
\end{equation}
for any $X \in \Gamma(TM)$. Furthermore, $(M,g)$ is called minimal if
\begin{equation} \label{eq5}
    \textrm{Tr} \ A^{{\circ}}_N=0.
\end{equation}

For more details related to hypersurfaces and submanifolds, we refer to \cite{Chen}.

\section{Almost Product-like Riemannian Manifolds \label{product-like}}

\begin{definition}
 If the Riemannian manifold $(\widetilde{M},\widetilde{g})$ with the almost product structure $F$ which has another tensor field $F^{\ast}$ of type $(1, 1)$ satisfying
\begin{equation} \label{eq6}
 \widetilde{g}(FX,Y)=\widetilde{g}(X,F^{\ast}Y)
\end{equation}
 for any  $X,Y\in \Gamma(T\widetilde{M})$, then the triple $(\widetilde{M},\widetilde{g},F)$ is called an almost product-like Riemannian manifold.
\end{definition}

 We see that $(F^{\ast})^{\ast}=F$, $(F^{\ast})^2=I$ and
\begin{equation} \label{eq7}
 \widetilde{g}(FX,F^{\ast}Y)=\widetilde{g}(X,Y).
\end{equation}

\begin{lemma} \label{Takano:lemma:1}
 The triple $(\widetilde{M},\widetilde{g},F)$ is an almost product-like Riemannian manifold if and only if so is $(\widetilde{M},\widetilde{g},F^{\ast})$.
\end{lemma}

\begin{example} \label{ex1}
Let $\mathbb{R}^{4}$  be a smooth manifold with local coordinate system $(x_{1},x_{2},x_{3},x_{4})$, which admits
the following almost product structure $F$
\begin{equation*}
 F=\begin{pmatrix}
 0&0  &1  & 0 \\
0 & 0 & 0 & 1 \\
1 & 0 & 0 & 0  \\
0 & 1 & 0 & 0 \\
\end{pmatrix}.
\end{equation*}
Thus the pair $(\mathbb{R}^{4}, F)$ is an almost product manifold. If we put
\begin{equation*}
 \widetilde{g}=\begin{pmatrix}
 1+e^{-x_{1}+x_{3}}&0  &0  & 0 \\
0 & e^{x_{1}-x_{3}} & 0 & 0 \\
0 & 0 & e^{-x_{1}+x_{3}} & 0  \\
0 & 0 & 0 & e^{x_{1}-x_{3}} \\
\end{pmatrix}
\end{equation*}
and
\begin{equation*}
F^{\ast}=\begin{pmatrix}
0&0  &(1+e^{x_{1}-x_{3}})^{-1}  & 0 \\
0 & 0 & 0 & 1 \\
1+e^{x_{1}-x_{3}} & 0 & 0 & 0  \\
0 & 1 & 0 & 0 \\
\end{pmatrix},
\end{equation*}
then the triple $(\mathbb{R}^{4},\widetilde{g},F)$   is an almost product-like Riemannian manifold and so is $(\mathbb{R}^{4},\widetilde{g},F^{\ast})$.
\end{example}
\begin{definition}
 If $F$ is parallel with respect to $\widetilde{\nabla}$, then $(\widetilde{M}, \widetilde{g},\widetilde{\nabla}, F)$ is called a locally product-like statistical manifold.
\end{definition}

From (\ref{eq6}), we get
\begin{eqnarray*}
\widetilde{g}\left((\widetilde{\nabla}_{X}F)Y,Z\right)=\widetilde{g}\left(Y,(\widetilde{\nabla}^{\ast}_{X}F^{\ast})Z\right).
\end{eqnarray*}
Hence we have
\begin{lemma} \label{Takano:lemma:2}
 $(\widetilde{M}, \widetilde{g},\widetilde{\nabla}, F)$ is a locally product-like statistical manifold if and only if so is $(\widetilde{M}, \widetilde{g},\widetilde{\nabla}^{\ast}, F^{\ast})$.
\end{lemma}
\begin{lemma} \label{Takano:lemma:3}
Let $(\widetilde{M}, \widetilde{g},\widetilde{\nabla}, F)$ be a locally product-like statistical manifold. If $\widetilde{M}$ is of constant curvature $\widetilde{c}$, then $\widetilde{c}=0$, that is, $\widetilde{M}$ is flat.
\end{lemma}
\begin{proof}
From $\widetilde{R}(X,Y)FZ=F\widetilde{R}(X,Y)Z$, we find
\begin{equation*}
    \widetilde{c} \left\{ \widetilde{g}(Y,FZ)X-\widetilde{g}(X,FZ)Y\right\}=\widetilde{c} \left\{ \widetilde{g}(Y,Z)FX-\widetilde{g}(X,Z)FY\right\}.
\end{equation*}
If $\widetilde{c}\neq 0$, then we get $mFX-(\text{tr} F)X=0$, where $m=\textrm{dim} \widetilde{M}$ and $\textrm{tr}F=\sum\limits _{\alpha=1}^{m}\widetilde{g}(Fe_{\alpha},e_{\alpha})$
for local orthonormal basis $\{e_{1},\ldots,e_{m}\}$ of $\widetilde{M}$.  It is clear that $\text{tr} F=\mp m$, which yields that $F=\mp I$ holds. Thus we obtain $\widetilde{c}=0$.
\end{proof}

\begin{example} \label{exx2}
 Let $(\mathbb{R}^{4},\widetilde{g},F)$ be an almost product-like Riemannian manifold of Example \ref{ex1}. We put the affine connection $\widetilde{\nabla}$ as follows:
\begin{align*}
{\widetilde{\nabla}_{\partial_{1}}} {\partial_{1}}&=\widetilde{\nabla}_{\partial_{1}} \partial_{3}=\widetilde{\nabla}_{\partial_{3}} \partial_{1}=\widetilde{\nabla}_{\partial_{3}} \partial_{3}=e^{-x_{1}+x_{3}}(\partial_{1}+\partial_{3}),  \\
{\widetilde{\nabla}_{\partial_{1}}} {\partial_{2}}&={\widetilde{\nabla}_{\partial_{2}}} {\partial_{1}}={\widetilde{\nabla}_{\partial_{3}}} {\partial_{4}}={\widetilde{\nabla}_{\partial_{4}}} {\partial_{3}}=-e^{-x_{1}+x_{3}}\partial_{2}-(1+e^{-x_{1}+x_{3}})\partial_{4},\\
{\widetilde{\nabla}_{\partial_{1}}} {\partial_{4}}&={\widetilde{\nabla}_{\partial_{4}}} {\partial_{1}}={\widetilde{\nabla}_{\partial_{2}}} {\partial_{3}}={\widetilde{\nabla}_{\partial_{3}}} {\partial_{2}}=-(1+e^{-x_{1}+x_{3}})\partial_{2}-e^{-x_{1}+x_{3}}\partial_{4},\\
{\widetilde{\nabla}_{\partial_{2}}} {\partial_{2}}&={\widetilde{\nabla}_{\partial_{4}}} {\partial_{4}}=-e^{x_{1}-x_{3}}\partial_{1}+\partial_{2}-e^{x_{1}-x_{3}}\partial_{3}-\partial_{4},\\
{\widetilde{\nabla}_{\partial_{2}}} {\partial_{4}}&={\widetilde{\nabla}_{\partial_{4}}} {\partial_{2}}=-e^{x_{1}-x_{3}}\partial_{1}-\partial_{2}-e^{x_{1}-x_{3}}\partial_{3}+\partial_{4}.
\end{align*}
Then we find
\begin{align*}
 {\widetilde{\nabla}^{\ast}_{\partial_{1}}}{\partial_{1}}&=-\frac{e^{-x_{1}+x_{3}}(2+e^{-x_{1}+x_{3}})}{1+e^{-x_{1}+x_{3}}}{\partial_{1}}-(1+e^{-x_{1}+x_{3}}){\partial_{3}},\\
{\widetilde{\nabla}^{\ast}_{\partial_{1}}} {\partial_{2}}&={\widetilde{\nabla}^{\ast}_{\partial_{2}}} {\partial_{1}}={\widetilde{\nabla}^{\ast}_{\partial_{1}}} {\partial_{4}}={\widetilde{\nabla}^{\ast}_{\partial_{4}}} {\partial_{1}}=(1+e^{-x_{1}+x_{3}})(\partial_{2}+\partial_{4}),\\
{\widetilde{\nabla}^{\ast}_{\partial_{1}}} {\partial_{3}}&={\widetilde{\nabla}^{\ast}_{\partial_{3}}} {\partial_{1}}=-\frac{e^{-2(x_{1}-x_{3})}}{1+e^{-x_{1}+x_{3}}}{\partial_{1}}-(1+e^{-x_{1}+x_{3}}){\partial_{3}}, \\
{\widetilde{\nabla}^{\ast}_{\partial_{2}}}{\partial_{2}}&={\widetilde{\nabla}^{\ast}_{\partial_{4}}}{\partial_{4}}=\frac{1}{1+e^{-x_{1}+x_{3}}}{\partial_{1}}-{\partial_{2}}+\frac{1+e^{-x_{1}+x_{3}}}{e^{-2(x_{1}-x_{3})}}{\partial_{3}}+{\partial_{4}},  \\
{\widetilde{\nabla}^{\ast}_{\partial_{2}}} {\partial_{3}}&={\widetilde{\nabla}^{\ast}_{\partial_{3}}} {\partial_{2}}={\widetilde{\nabla}^{\ast}_{\partial_{3}}}{\partial_{4}}={\widetilde{\nabla}^{\ast}_{\partial_{4}}} {\partial_{3}}=e^{-x_{1}+x_{3}}({\partial_{2}+{\partial_{4}}}),   \\
{\widetilde{\nabla}^{\ast}_{\partial_{2}}}{\partial_{4}}&={\widetilde{\nabla}^{\ast}_{\partial_{4}}}{\partial_{2}}={e^{x_{1}-x_{3}}}{\partial_{1}}+{\partial_{2}}+{e^{x_{1}-x_{3}}}{\partial_{3}}-{\partial_{4}},   \\
{\widetilde{\nabla}^{\ast}_{\partial_{3}}}{\partial_{3}}&=-\frac{e^{-2(x_{1}-x_{3})}}{1+e^{-x_{1}+x_{3}}}{\partial_{1}}+(1-e^{-x_{1}+x_{3}}){\partial_{3}}.
\end{align*}
Therefore $(\mathbb{R}^{4},\widetilde{g},F,\widetilde{\nabla})$  is a locally product-like statistical manifold and
so is $(\mathbb{R}^{4},\widetilde{g},F^{\ast},\widetilde{\nabla}^{\ast})$.
\end{example}

Next, let $(M,g)$ be a submanifold of the statistical manifold $\widetilde{M}$. The Gauss and Weingarten formulas with respect to affine connections $\widetilde{\nabla}$ and $\widetilde{\nabla}^{\ast}$ are denoted by
\begin{align*}
\widetilde{\nabla}_{X}Y&=\nabla_{X}Y+\sigma(X,Y),  \\
\widetilde{\nabla}_{X}V&=-A_{V}X+D_{X}V,  \\
\widetilde{\nabla}_{X}^{\ast}Y&=\nabla_{X}^{\ast}Y+\sigma^{\ast}(X,Y),  \\
\widetilde{\nabla}_{X}^{\ast}V&=-A_{V}^{\ast}X+D_{X}^{\ast}V
\end{align*}
for any vector fields $X,Y$ tangent to $M$ and $V$ normal to $M$, respectively. Because of $\widetilde{\nabla}$ (resp. $\widetilde{\nabla}^{\ast}$) is torsion-free, the affine connection $\nabla$ (resp. $\nabla^{\ast}$) is torsion-free and the second fundamental form $\sigma$ (resp. $\sigma^{\ast}$) is symmetric. Affine connections $\nabla$ and $\nabla^{\ast}$ are dual each other. The triple $(M,g,\nabla)$ (resp. $(M,g,\nabla^{\ast})$) is a statistical submanifold of $\widetilde{M}$. Also, $D$ and $D^{\ast}$ are dual affine connections each other. From $(\widetilde{\nabla}^{\ast})^{\ast}=\widetilde{\nabla}$, we find $(\nabla^{\ast})^{\ast}=\nabla$, $(\sigma^{\ast})^{\ast}=\sigma$, $(A^{\ast})^{\ast}=A$ and $(D^{\ast})^{\ast}=D$. If $\sigma=0$ (resp. $\sigma^{\ast}=0$), then $M$ is said to be totally geodesic with respect to $\widetilde{\nabla}$ (resp. $\widetilde{\nabla}^{\ast}$). H. Furuhata and  I. Hasegawa \cite{Furuhata-2016} studied statistical submanifolds.

\begin{lemma} \label{Takano:lemma:4}
For any vector fields $X,Y$ tangent to $M$ and $V$ normal to $M$, we get
\begin{align*}
    \widetilde{g}(\sigma(X,Y),V)&=g(Y,A^{\ast}_{V}X), \\
     \widetilde{g}(\sigma^{\ast}(X,Y),V)&=g(Y,A_{V}X).
\end{align*}
The second fundamental form $\sigma$ (resp. $\sigma^{\ast}$) vanishes zero if and only if the shape operator $A^{\ast}$ (resp. $A$) so is.
\end{lemma}
\begin{corollary} \label{Takano:cor:1}
 We find
 \begin{align*}
     g(A^{\ast}_{V}X,Y)&=g(X,A^{\ast}_{V}Y), \\
     g(A_{V}X,Y)&=g(X,A_{V}Y).
 \end{align*}
\end{corollary}

For any vector field $X$ tangent to $M$, we put (\cite{Yano})
\begin{align*}
FX = fX + hX,  \quad \quad F^{\ast}X=f^{\ast}X+h^{\ast}X,
\end{align*}
where $fX$ (resp. $f^{\ast}X$) is the tangent part of $FX$ (resp. $F^{\ast}X$) and $hX$ (resp. $h^{\ast}X$) the normal part of $FX$ (resp. $F^{\ast}X$). For any vector field $V$ normal to $M$ we set
\begin{align*}
FV = tV + sV,  \quad \quad F^{\ast}V=t^{\ast}V+s^{\ast}V,
\end{align*}
where $tV$ (resp. $t^{\ast}V$) is the tangent part of $FV$ (resp. $F^{\ast}V$) and $sV$ (resp. $s^{\ast}V$) the normal part of $FV$ (resp. $F^{\ast}V$). From $(F^{\ast})^{\ast}=F$, $F^{2}=I$ and $(F^{\ast})^{2}=I$, we get
\begin{equation*}
 (f^{\ast})^{\ast}=f, \quad \quad \quad (h^{\ast})^{\ast}=h, \quad \quad \quad (t^{\ast})^{\ast}=t, \quad \quad \quad (s^{\ast})^{\ast}=s,
\end{equation*}
\begin{equation*}
 f^{2}X=X-thX, \quad hfX+shX=0, \quad ftV+tsV=0, \quad s^{2}V=V-htV
\end{equation*}
and
\begin{align*}
(f^{\ast})^{2}X&=X-t^{\ast}h^{\ast}X, & h^{\ast}f^{\ast}X+s^{\ast}h^{\ast}X&=0, \\
f^{\ast}t^{\ast}V+t^{\ast}s^{\ast}V&=0,  & (s^{\ast})^{2}V&=V-h^{\ast}t^{\ast}V.
\end{align*}
Because of $\widetilde{g}(FE,G)=\widetilde{g}(E,F^{\ast}G)$ for any vector fields $E$ and $G$ on $\widetilde{M}$, we find
\begin{align*}
g(fX,Y)&=g(X,f^{\ast}Y),  &  g(fX,f^{\ast}Y)&=g(X,Y)-\widetilde{g}(hX,h^{\ast}Y), \\
\widetilde{g}(hX,V)&=g(X,t^{\ast}V), & \widetilde{g}(h^{\ast}X,V)&=g(X,tV), \\
\widetilde{g}(sU,V)&=g(U,s^{\ast}V),  & \widetilde{g}(sU,s^{\ast}V)&=\widetilde{g}(U,V)-g(tU,t^{\ast}V).
\end{align*}
Thus $f$ (resp. $s$) vanishes identically if and only if $f^{\ast}$ (resp. $s^{\ast}$) so is, and $h$ (resp. $t$) vanishes identically is and only if $t^{\ast}$ (resp. $h^{\ast}$) so is.
Hence we have
\begin{lemma} \label{Takano:lemma:5}
We find for each $p \in M$
\begin{enumerate}
\item[(1)] $F\left(T_{p}M\right)\subset T_{p}M$ if and only if  $F^{\ast}\left(T_{p}M\right)^{\bot}\subset \left(T_{p}M\right)^{\bot}$.

\item[(2)] $F\left(T_{p}M\right)\subset \left(T_{p}M\right)^{\bot}$ if and only if  $F^{\ast}\left(T_{p}M\right)\subset \left(T_{p}M\right)^{\bot}$.

\item[(3)] $F\left(T_{p}M\right)^{\bot}\subset T_{p}M$ if and only if  $F^{\ast}\left(T_{p}M\right)^{\bot}\subset T_{p}M$.

\item[(4)] $F\left(T_{p}M\right)^{\bot}\subset \left(T_{p}M\right)^{\bot}$ if and only if  $F^{\ast}\left(T_{p}M\right)\subset T_{p}M$.
\end{enumerate}
\end{lemma}

If $F\left(T_{p}M\right)\subset T_{p}M$ (resp. $F^{\ast}\left(T_{p}M\right)\subset T_{p}M$) for each $p\in M$, then $M$ is said to be $F$-invariant
(resp. $F^{\ast}$-invariant) in $\widetilde{M}$. Then $h$ and $t^{\ast}$ (resp. $h^{\ast}$ and $t$) vanish identically. because of $f^{2}=I$
and $g(fX,f^{\ast}Y)=g(X,Y)$, the triple $(M,g,f)$ is an almost product-like Riemannian manifold. Hence we
have

\begin{theorem}\label{Takano:thm:1}
Let $M$ be a submanifold of an almost product-like Riemannian manifold $\widetilde{M}$. A necessary and sufficient condition for $M$
to be $F$-invariant or $F^{\ast}$-invariant is that the triple $(M,g,f)$ is an almost product-like Riemannian manifold.
\end{theorem}

\begin{example}\label{exx3}
Let $(\mathbb{R}^{4},\widetilde{g},F)$ be an almost product-like Riemannian manifold of Example \ref{ex1}.
We define the immersion $\widetilde{f}: M\rightarrow \mathbb{R}^{4}$ by $\widetilde{f}(x_{1},x_{3})=(x_{1},x_{2},x_{3},x_{4})$.
Then the induced metric $g$ is given by
\begin{eqnarray*}
g=\left(
    \begin{array}{cc}
      1+e^{-x_{1}+x_{3}} & 0 \\
      0 & e^{-x_{1}+x_{3}} \\
    \end{array}
  \right).
\end{eqnarray*}
We find
\begin{eqnarray*}
f=\left(
    \begin{array}{cccc}
      0 & 0 & 1 & 0 \\
      0 & 0 & 0 & 0 \\
      1 & 0 & 0 & 0 \\
      0 & 0 & 0 & 0 \\
    \end{array}
  \right), \ \ h=O,  \ \ \ t=O, \ \ \
  s= \left(
    \begin{array}{cccc}
      0 & 0 & 0 & 0 \\
      0 & 0 & 0 & 1 \\
      0 & 0 & 0 & 0 \\
      0 & 1 & 0 & 0 \\
    \end{array}
  \right)
\end{eqnarray*}
and

\begin{eqnarray*}
f^{\ast}=\left(
    \begin{array}{cccc}
      0 & 0 & \left(1+e^{x_{1}-x_{3}}\right)^{-1} & 0 \\
      0 & 0 & 0 & 0 \\
      1+e^{x_{1}-x_{3}} & 0 & 0 & 0 \\
      0 & 0 & 0 & 0 \\
    \end{array}
  \right), \ \ h^{\ast}=O,  \
\end{eqnarray*}
\begin{eqnarray*}
t^{\ast}=O, \ s^{\ast}= \left(
    \begin{array}{cccc}
      0 & 0 & 0 & 0 \\
      0 & 0 & 0 & 1 \\
      0 & 0 & 0 & 0 \\
      0 & 1 & 0 & 0 \\
    \end{array}
  \right).
\end{eqnarray*}
Therefore the triple $(M,g,f)$ is $F$-invariant and $F^{\ast}$-invariant of $\mathbb{R}^{4}$.
\end{example}

Moreover, we get
\begin{align*}
g((\nabla_{Z}f)X,Y)&=g(X,(\nabla^{\ast}_{Z}f^{\ast})Y), \\
\widetilde{g}((D_{X}s)U,V)&=\widetilde{g}(U,(D^{\ast}_{X}s^{\ast})V).
\end{align*}
Thus we have
\begin{lemma} \label{Takano:lemma:6}
The structure $f$ (resp. $s$) is parallel with respect to $\nabla$ (resp. $D$) if and only if $f^{\ast}$ (resp. $s^{\ast}$) is parallel with respect to $\nabla^{\ast}$ (resp. $D^{\ast}$).
\end{lemma}

\begin{lemma} \label{Takano:lemma:7}
 In the locally product-like statistical manifold, we get
 \begin{align*}
 (\nabla_{X}f)Y-A_{hY}X-t(\sigma(X,Y))&=0, \\
 (\overline{D}_{X}h)Y+\sigma(X,fY)-s(\sigma(X,Y))&=0, \\
 (\overline{\nabla}_{X}t)V-A_{sV}X+f(A_{V}X)&=0, \\
 (D_{X}s)V+\sigma(X,tV)+h(A_{V}X)&=0,
 \end{align*}
 where we put
 \begin{align*}
  (\nabla_{X}f)Y&=\nabla_{X}(fY)-f(\nabla_{X}Y),    & (\overline{D}_{X}h)Y&=D_{X}(hY)-h(\nabla_{X}Y),\\
  (\overline{\nabla}_{X}t)V&=\nabla_{X}(tV)-t(D_{X}V),    &  (D_{X}s)V&=D_{X}(sV)-s(D_{X}V).
 \end{align*}
\end{lemma}

\begin{proof}
We find
\begin{eqnarray*}
\left(\widetilde{\nabla}_{X}F\right)Y&=&\left(\nabla _{X}f\right)Y-A_{hY}X-t\left(\sigma(X,Y)\right)+
\left(\overline{D}_{X}h\right)Y \nonumber \\
&&+\sigma(X,fY)-s(\sigma(X,Y)), \nonumber \\
\left(\widetilde{\nabla}_{X}F\right)V&=&\left(\overline{\nabla} _{X}t\right)V-A_{sV}X-f\left(A_{V}X\right)+
\left(D_{X}s\right)V \nonumber \\
&&+\sigma(X,tV)+h(A_{V}X).
\end{eqnarray*}
From $\widetilde{\nabla}F=0$, we have the results.
\end{proof}

\begin{corollary} \label{Takano:cor:2}
In the locally product-like statistical manifold, we get
\begin{align*}
\left(\nabla^{\ast}_{X}f^{\ast}\right)Y-A^{\ast}_{h^{\ast}Y}X-t^{\ast}\left(\sigma^{\ast}(X,Y)\right)&=0, \nonumber \\
\left(\overline{D}^{\ast}_{X}h^{\ast}\right)Y+\sigma^{\ast}(X,f^{\ast}Y)-s^{\ast}\left(\sigma^{\ast}(X,Y)\right)&=0, \nonumber \\
\left(\overline{\nabla}^{\ast}_{X}t^{\ast}\right)V-A^{\ast}_{s^{\ast}V}X+f^{\ast}(A^{\ast}_{V}X)&=0,\nonumber \\
\left(D^{\ast}_{X}s^{\ast}\right)V+\sigma^{\ast}(X,t^{\ast}V)+h^{\ast}(A^{\ast}_{V}X)&=0,
\end{align*}
where we put
\begin{align*}
\left(\nabla^{\ast}_{X}f^{\ast}\right)Y&=\nabla^{\ast}_{X}(f^{\ast}Y)-f^{\ast}(\nabla^{\ast}_{X}Y), &
\left(\overline{D}^{\ast}_{X}h^{\ast}\right)Y&= D^{\ast}_{X}(h^{\ast}Y)-h^{\ast}(\nabla^{\ast}_{X}Y), \\
\left(\overline{\nabla}^{\ast}_{X}t^{\ast}\right)V&=\nabla^{\ast}_{X}(t^{\ast}V)-t^{\ast}(D^{\ast}_{X}V), &
\left(D^{\ast}_{X}s^{\ast}\right)V&=D^{\ast}_{X}(s^{\ast}V)-s^{\ast}(D^{\ast}_{X}V).
\end{align*}
\end{corollary}

\begin{example} \label{exx4}
$(\mathbb{R}^{4},\widetilde{g},F,\widetilde{\nabla})$ be a locally product-like statistical manifold of Example \ref{exx2}
and $(M,g)$ be a statistical submanifold of Example \ref{exx3}. If $\partial_{1}$, $\partial_{3}$ is tangent to $M$ and $\partial_{2}$, $\partial_{4}$ is normal to $M$, then we find
\begin{align*}
\nabla_{\partial_{i}}\partial_{j}&=e^{-x_{1}+x_{3}}(\partial_{1}+\partial_{3}), \ \ \ \sigma(\partial_{i},\partial_{j})=0, \ \  (i,j\in \{1,3\}), \\
A_{\partial_{2}}\partial_{1}&=A_{\partial_{4}}\partial_{3}=0, \ \ \ \
D_{\partial_{1}}\partial_{2}=D_{\partial_{3}}\partial_{4}=-e^{-x_{1}+x_{3}}\partial_{2}-(1+e^{-x_{1}+x_{3}})\partial_{4}, \\
A_{\partial_{2}}\partial_{3}&=A_{\partial_{4}}\partial_{1}=0,\ \ \ \
D_{\partial_{1}}\partial_{4}=D_{\partial_{3}}\partial_{2}=-(1+e^{-x_{1}+x_{3}})\partial_{2}-e^{-x_{1}+x_{3}}\partial_{4},
\end{align*}
and
\begin{align*}
\nabla^{\ast}_{\partial_{1}}\partial_{1}&=-\tfrac{e^{-x_{1}+x_{3}}(2+e^{-x_{1}+x_{3}})}{1+e^{-x_{1}+x_{3}}}\partial_{1}-(1+e^{-x_{1}+x_{3}})\partial_{3}, \
 \sigma^{\ast}(\partial_{1},\partial_{1})=0,  \\
\nabla^{\ast}_{\partial_{1}}\partial_{3}&=\nabla^{\ast}_{\partial_{3}}\partial_{1}=-\tfrac{e^{-2(x_{1}-x_{3})}}{1+e^{-x_{1}+x_{3}}}\partial_{1}
-(1+e^{-x_{1}+x_{3}})\partial_{3},  \  \sigma^{\ast}(\partial_{1},\partial_{3})=0, \\
 \nabla^{\ast}_{\partial_{3}}\partial_{3}&=-\tfrac{e^{-2(x_{1}-x_{3})}}{1+e^{-x_{1}+x_{3}}}\partial_{1}+(1-e^{-x_{1}+x_{3}})\partial_{3}, \ \ \ \ \ \ \ \ \ \ \ \ \
\sigma^{\ast}(\partial_{3},\partial_{3})=0, \\
A^{\ast}_{\partial_{2}}\partial_{1}&=A^{\ast}_{\partial_{4}}\partial_{1}=0, \ \  \ \ D^{\ast}_{\partial_{1}}\partial_{2}=D^{\ast}_{\partial_{1}}\partial_{4}=(1+e^{-x_{1}+x_{3}})(\partial_{2}+\partial_{4}), &
 & \\
A^{\ast}_{\partial_{2}}\partial_{3}&=A^{\ast}_{\partial_{4}}\partial_{3}=0, \ \ \ \ D^{\ast}_{\partial_{3}}\partial_{2}=D^{\ast}_{\partial_{3}}\partial_{4}=e^{-x_{1}+x_{3}}(\partial_{2}+\partial_{4}). & &
\end{align*}
Thus the triple $(M,g,\nabla)$ (resp. $(M,g,\nabla^{\ast})$) is totally geodesic relative to $\widetilde{\nabla}$ (resp. $\widetilde{\nabla}^{\ast}$).
Also, the structure $f$ (resp. $f^{\ast}$) is parallel with respect to $\nabla$ (resp. $\nabla^{\ast}$). Therefore $(M,g,f,\nabla)$ is a locally
product-like statistical manifold and so is $(M,g,f^{\ast},\nabla^{\ast})$.
\end{example}
\begin{example}\label{exx5}
Let $(\mathbb{R}^{4},\widetilde{g},F,\widetilde{\nabla})$ be a locally product-like statistical manifold of Example \ref{exx2}.
We define the immersion $\widetilde{f}:M\rightarrow \mathbb{R}^{4}$ by $\widetilde{f}(x_{1},x_{2},x_{3})=(x_{1},x_{2},x_{3},x_{4})$.
Then the induced metric $g$ is given by
\begin{equation*}
g=\begin{pmatrix}
1+e^{-x_{1}+x_{3}} & 0 & 0 \\
0 & e^{x_{1}-x_{3}} & 0 \\
0 & 0 & e^{-x_{1}+x_{3}} \\
\end{pmatrix}.
\end{equation*}
If $\partial_{1}, \partial_{2}, \partial_{3}$ is tangent to $M$ and $\partial_{4}$ is normal to $M$, then we find
\begin{align*}
  {\nabla}_{\partial_{i}} {\partial_{j}}&=e^{-x_{1}+x_{3}}(\partial_{1}+\partial_{3}),   &  \sigma(\partial_{i},\partial_{j})&=0, \ \ \  i,j \in \left\{ 1,3 \right\},\\
{\nabla}_{\partial_{1}} {\partial_{2}}&={\nabla}_{\partial_{2}} {\partial_{1}}=-e^{-x_{1}+x_{3}} \partial_{2},& \sigma(\partial_{1},\partial_{2})&=-(1+e^{-x_{1}+x_{3}}) \partial_{4}, \\
{\nabla}_{\partial_{2}} {\partial_{2}}&=-e^{x_{1}-x_{3}}\partial_{1}+\partial_{2}-e^{x_{1}-x_{3}}\partial_{3},   & \sigma(\partial_{2},\partial_{2})&=-\partial_{4}, \\
 {\nabla}_{\partial_{2}} {\partial_{3}}&={\nabla}_{\partial_{3}} {\partial_{2}}=-(1+e^{-x_{1}+x_{3}})\partial_{2},  & \sigma(\partial_{2},\partial_{3})&=-e^{-x_{1}+x_{3}} \partial_{4}, \\
{A_{\partial_{4}}}\partial_{1}&=(1+e^{-x_{1}+x_{3}})\partial_{2},   & {D_{\partial_{1}}}\partial_{4}&=e^{-x_{1}+x_{3}} \partial_{4}, \\
{A_{\partial_{4}}}\partial_{2}&=e^{x_{1}-x_{3}}\partial_{1}+\partial_{2}+e^{x_{1}-x_{3}}\partial_{3},  &  {D_{\partial_{2}}}\partial_{4}&=\partial_{4}, \\
{A_{\partial_{4}}}\partial_{3}&=e^{-x_{1}+x_{3}}\partial_{2},  &  {D_{\partial_{3}}}\partial_{4}&=-(1+e^{-x_{1}+x_{3}})\partial_{4}
\end{align*}
and

\begin{align*}
{\nabla}^{\ast}_{\partial_{1}}{\partial_{1}}&=-\dfrac{e^{-x_{1}+x_{3}}(2+e^{-x_{1}+x_{3}})}{1+e^{-x_{1}+x_{3}}}{\partial_{1}}-(1+e^{-x_{1}+x_{3}}){\partial_{3}}, \quad \sigma^{\ast}(\partial_{1},\partial_{1})=0, \\
{\nabla}^{\ast}_{\partial_{1}}{\partial_{2}}&={\nabla}^{\ast}_{\partial_{2}}{\partial_{1}}=(1+e^{-x_{1}+x_{3}})\partial_{2}, \quad \sigma^{\ast}(\partial_{1},\partial_{2})=(1+e^{-x_{1}+x_{3}})\partial_{4}, \\
{\nabla}^{\ast}_{\partial_{1}}{\partial_{3}}&={\nabla}^{\ast}_{\partial_{3}}{\partial_{1}}=-\frac{e^{-2(x_{1}-x_{3})}}{1+e^{-x_{1}+x_{3}}}\partial_{1}-(1+e^{-x_{1}+x_{3}})\partial_{3}, \sigma^{\ast}(\partial_{1},\partial_{3})=0,\\
{\nabla}^{\ast}_{\partial_{2}}{\partial_{2}}&=\frac{1}{1+e^{-x_{1}+x_{3}}}\partial_{1}-\partial_{2}+\frac{1+e^{-x_{1}+x_{3}}}{e^{-2(x_{1}-x_{3})}}\partial_{3}, \quad \sigma^{\ast}(\partial_{2},\partial_{2})=\partial_{4},
\end{align*}

\begin{align*}
{\nabla}^{\ast}_{\partial_{2}}\partial_{3}&= \nabla^{\ast}_{\partial_{3}}\partial_{2}=e^{-x_{1}+x_{3}}\partial_{2}, \quad \sigma^{\ast}(\partial_{2},\partial_{3})=e^{-x_{1}+x_{3}}\partial_{4}, \\
{\nabla}^{\ast}_{\partial_{3}}{\partial_{3}}&=-\frac{e^{-2(x_{1}-x_{3})}}{1+e^{-x_{1}+x_{3}}}\partial_{1}+(1-e^{-x_{1}+x_{3}}){\partial_{3}}, \quad \sigma^{\ast}(\partial_{3},\partial_{3})=0, \\
A^{\ast}_{\partial_{4}}\partial_{1}&=-(1+e^{-x_{1}+x_{3}})\partial_{2}, \quad D^{\ast}_{\partial_{1}} \partial_{4}=(1+e^{-x_{1}+x_{3}})\partial_{4}, \\
A^{\ast}_{\partial_{4}}\partial_{2}&=-e^{x_{1}-x_{3}}\partial_{1}-\partial_{2}-e^{x_{1}-x_{3}}\partial_{3}, \quad D^{\ast}_{\partial_{2}} \partial_{4}=-\partial_{4}, \\
A^{\ast}_{\partial_{4}}\partial_{3}&=-e^{-x_{1}+x_{3}}\partial_{2},  \quad D^{\ast}_{\partial_{3}}\partial_{4}=e^{-x_{1}+x_{3}}\partial_{4}.
\end{align*}
\normalsize
Thus the triple $(M,g,\nabla)$ is a statistical hypersurface of $\mathbb{R}^{4}$ and so is $(M,g,\nabla^{\ast})$.
\end{example}

For any vector fields $X,Y,Z$ tangent to $M$, we get
\begin{equation*}
 \widetilde{R}(X,Y)Z=R(X,Y)Z-A_{\sigma(Y,Z)}X+A_{\sigma(X,Z)}Y+(D_{X}\sigma)(Y,Z)-(D_{Y}\sigma)(X,Z),
\end{equation*}
\begin{align*}
 \widetilde{R}^{\ast}(X,Y)Z&=R^{\ast}(X,Y)Z-A^{\ast}_{\sigma^{\ast}(Y,Z)}X+A^{\ast}_{\sigma^{\ast}(X,Z)}Y+(D^{\ast}_{X}\sigma^{\ast})(Y,Z)\\
 &\quad -(D^{\ast}_{Y}\sigma^{\ast})(X,Z),
\end{align*}
where we put
\begin{equation*}
(D_{X}\sigma)(Y,Z)=D_{X}\{\sigma(Y,Z)\}-\sigma(\nabla_{X}Y,Z)-\sigma(Y,\nabla_{X}Z),
\end{equation*}
\begin{equation*}
(D^{\ast}_{X}\sigma^{\ast})(Y,Z)=D^{\ast}_{X}\{\sigma^{\ast}(Y,Z)\}-\sigma^{\ast}(\nabla^{\ast}_{X}Y,Z)-\sigma^{\ast}(Y,\nabla^{\ast}_{X}Z).
\end{equation*}
Hence we have the equation of Gauss relative to $\widetilde{\nabla}$ and the equation of Codazzi relative to $\widetilde{\nabla}$ \cite{Furuhata-2016}.

\begin{proposition}\label{Takano:prop:1a}
 For any vector field $W$ tangent to $M$, we get
 \begin{eqnarray*}
  \widetilde{g}(\widetilde{R}(X,Y)Z,W)&=&g(R(X,Y)Z,W)-\widetilde{g}(\sigma(Y,Z),\sigma^{\ast}(X,W))\\
   &&+\widetilde{g}(\sigma(X,Z),\sigma^{\ast}(Y,W)), \\
   (\widetilde{R}(X,Y)Z)^{\perp}&=&(D_{X}\sigma)(Y,Z)-(D_{Y}\sigma)(X,Z).
 \end{eqnarray*}
\end{proposition}
Also, we have the equation of Gauss relative to $\widetilde{\nabla}^{\ast}$ and the equation of Codazzi relative to $\widetilde{\nabla}^{\ast}$.
\begin{proposition}\label{Takano:prop:1b}
 For any vector field $W$ tangent to $M$, we get
 \begin{eqnarray*}
   \widetilde{g}(\widetilde{R}^{\ast}(X,Y)Z,W)&=&g(R^{\ast}(X,Y)Z,W)-\widetilde{g}(\sigma^{\ast}(Y,Z),\sigma(X,W))\\
   &&+\widetilde{g}(\sigma^{\ast}(X,Z),\sigma(Y,W)), \\
   (\widetilde{R}^{\ast}(X,Y)Z)^{\perp}&=&(D^{\ast}_{X}\sigma^{\ast})(Y,Z)-(D^{\ast}_{Y}\sigma^{\ast})(X,Z).
 \end{eqnarray*}
\end{proposition}

For any vector fields $X, Y$ tangent to $M$ and  $V$ normal to $M$, we obtain
\begin{align*}
\widetilde{R}(X,Y)V&=-(\nabla_{X}A)_{V}Y+(\nabla_{Y}A)_{V}X+R^{\perp}(X,Y)V-\sigma(X,A_{V}Y)\\
& \quad +\sigma(Y,A_{V}X),
\end{align*}
\begin{align*}
\widetilde{R}^{\ast}(X,Y)V&=-(\nabla^{\ast}_{X}A^{\ast})_{V}Y+(\nabla^{\ast}_{Y}A^{\ast})_{V}X+(R^{\ast})^{\perp}(X,Y)V-\sigma^{\ast}(X,A^{\ast}_{V}Y)\\
&\quad+\sigma^{\ast}(Y,A^{\ast}_{V}X),
\end{align*}
where we put
\begin{eqnarray*}
 (\nabla_{X}A)_{V}Y&=&\nabla_{X}(A_{V}Y)-A_{D_{X}V}Y-A_{V}(\nabla_{X}Y),  \\
R^{\perp}(X,Y)V&=&D_{X}(D_{Y}V)-D_{Y}(D_{X}V)-D_{[X,Y]}V, \\
(\nabla^{\ast}_{X}A^{\ast})_{V}Y&=&\nabla^{\ast}_{X}(A^{\ast}_{V}Y)-A^{\ast}_{D^{\ast}_{X}V}Y-A^{\ast}_{V}(\nabla^{\ast}_{X}Y),  \\
(R^{\ast})^{\perp}(X,Y)V&=&D^{\ast}_{X}(D^{\ast}_{Y}V)-D^{\ast}_{Y}(D^{\ast}_{X}V)-D^{\ast}_{[X,Y]}V. \\
\end{eqnarray*}
Hence we have equation of Ricci relative to $\widetilde{\nabla}$.

\begin{proposition} \label{Takano:prop:1c}
 For any vector field $U$ normal to $M$, we get
 \begin{equation*}
    \widetilde{g}(\widetilde{R}(X,Y)V,U)=\widetilde{g}(R^{\perp}(X,Y)V,U)+g([A^{\ast}_{U},A_{V}]X,Y ).
 \end{equation*}
\end{proposition}
Also  we have the equation Ricci relative to $\widetilde{\nabla}^{\ast}$.
\begin{proposition} \label{Takano:prop:1d}
 For any vector field $U$ normal to $M$, we get
 \begin{equation*}
    \widetilde{g}(\widetilde{R}^{\ast}(X,Y)V,U)=\widetilde{g}((R^{\ast})^{\perp}(X,Y)V,U)+g([A_{U},A^{\ast}_{V}]X,Y ).
 \end{equation*}
\end{proposition}

From $\widetilde{R}(X,Y)FZ=F\widetilde{R}(X,Y)Z$ and $\widetilde{R}(X,Y)FV=F\widetilde{R}(X,Y)V$, we find

\begin{proposition} \label{Takano:prop:1e}
 Let $(M,g)$  be the submanifold of the locally product-like statistical manifold $\widetilde{M}$. For any vector fields $X,Y,Z$ tangent to $M$ and any vector field $V$ normal to $M$, we have
\begin{eqnarray*}
&&R(X,Y)fZ-A_{\sigma(Y,fZ)}X+A_{\sigma(X,fZ)}Y-(\nabla_{X}A)_{hZ}Y+(\nabla_{Y}A)_{hZ}X \\
&&=f(R(X,Y)Z)-f(A_{\sigma(Y,Z)}X)+f(A_{\sigma(X,Z)}Y)+t\left\{(D_{X}\sigma)(Y,Z)\right.\\
&&\left. \quad-(D_{Y}\sigma)(X,Z) \right\}, \\
&&R^{\perp}(X,Y)hZ-\sigma(X,A_{hZ}Y)+\sigma(Y,A_{hZ}X)+(D_{X}\sigma)(Y,fZ)\\
&& \quad -(D_{Y}\sigma)(X,fZ)=h(R(X,Y)Z)-h(A_{\sigma(Y,Z)}X)+h(A_{\sigma(X,Z)}Y)\\
&& \quad +s\{(D_{X}\sigma)(Y,Z)-(D_{Y}\sigma)(X,Z) \}, \\
&&R(X,Y)tV-A_{\sigma(Y,tV)}X+A_{\sigma(X,tV)}Y-(\nabla_{X}A)_{sV}Y+(\nabla_{Y}A)_{sV}X \\
&&=-f((\nabla_{X}A)_{V}Y)+f((\nabla_{Y}A)_{V}X)+t(R^{\perp}(X,Y)V)-t(\sigma(X,A_{V}Y))\\
&&\quad +t(\sigma(Y,A_{V}X)),\\
&&R^{\perp}(X,Y)sV-\sigma(X,A_{sV}Y)+\sigma(Y,A_{sV}X)+(D_{X}\sigma)(Y,tV)\\
&&-(D_{Y}\sigma)(X,tV)=s(R^{\perp}(X,Y)V)-s(\sigma(X,A_{V}Y))+s(\sigma(Y,A_{V}X))\\
&&-h((\nabla_{X}A)_{V}Y)+h((\nabla_{Y}A)_{V}X).
 \end{eqnarray*}
\end{proposition}

\begin{corollary}
 Let $(M,g)$  be the submanifold of the locally product-like statistical manifold $\widetilde{M}$. For any vector fields $X,Y,Z$ tangent to $M$ and any vector field $V$ normal to $M$, we have
 \begin{eqnarray*}
&&R^{\ast}(X,Y)f^{\ast}Z-A^{\ast}_{\sigma^{\ast}(Y,f^{\ast}Z)}X+A^{\ast}_{\sigma^{\ast}(X,f^{\ast}Z)}Y-(\nabla^{\ast}_{X}A^{\ast})_{h^{\ast}Z}Y\\
&&+(\nabla^{\ast}_{Y}A^{\ast})_{h^{\ast}Z}X =f^{\ast}(R^{\ast}(X,Y)Z)-f^{\ast}(A^{\ast}_{\sigma^{\ast}(Y,Z)}X)+f^{\ast}(A^{\ast}_{\sigma^{\ast}(X,Z)}Y)\\
&&+t^{\ast}\{(D^{\ast}_{X}\sigma^{\ast})(Y,Z)-(D^{\ast}_{Y}\sigma^{\ast})(X,Z) \},\\
&&(R^{\ast})^{\perp}(X,Y)h^{\ast}Z-\sigma^{\ast}(X,A^{\ast}_{h^{\ast}Z}Y)+\sigma^{\ast}(Y,A^{\ast}_{h^{\ast}Z}X)+(D^{\ast}_{X}\sigma^{\ast})(Y,f^{\ast}Z)\\
&&-(D^{\ast}_{Y}\sigma^{\ast})(X,f^{\ast}Z)=h^{\ast}(R^{\ast}(X,Y)Z)-h^{\ast}(A^{\ast}_{\sigma^{\ast}(Y,Z)}X)+h^{\ast}(A^{\ast}_{\sigma^{\ast}(X,Z)}Y) \\
&&+s^{\ast}\{(D^{\ast}_{X}\sigma^{\ast})(Y,Z)-(D^{\ast}_{Y}\sigma^{\ast})(X,Z) \}, \\
&&R^{\ast}(X,Y)t^{\ast}V-A^{\ast}_{\sigma^{\ast}(Y,t^{\ast}V)}X+A^{\ast}_{\sigma^{\ast}(X,t^{\ast}V)}Y-(\nabla^{\ast}_{X}A^{\ast})_{s^{\ast}V}Y\\
&&+(\nabla^{\ast}_{Y}A^{\ast})_{s^{\ast}V}X=-f^{\ast}((\nabla^{\ast}_{X}A^{\ast})_{V}Y)+f^{\ast}((\nabla^{\ast}_{Y}A^{\ast})_{V}X) \\
&&+t^{\ast}((R^{\perp})^{\ast}(X,Y)V)-t^{\ast}(\sigma^{\ast}(X,A^{\ast}_{V}Y))+t^{\ast}(\sigma^{\ast}(Y,A^{\ast}_{V}X)),\\
&&(R^{\ast})^{\perp}(X,Y)s^{\ast}V-\sigma^{\ast}(X,A^{\ast}_{s^{\ast}V}Y)+\sigma^{\ast}(Y,A^{\ast}_{s^{\ast}V}X)+(D^{\ast}_{X}\sigma^{\ast})(Y,t^{\ast}V)\\
&&-(D^{\ast}_{Y}\sigma^{\ast})(X,t^{\ast}V)=s^{\ast}((R^{\ast})^{\perp}(X,Y),V)-s^{\ast}(\sigma^{\ast}(X,A^{\ast}_{V}Y))\\
&&+s^{\ast}(\sigma^{\ast}(Y,A^{\ast}_{V}X))-h^{\ast}((\nabla^{\ast}_{X}A^{\ast})_{V}Y)+h^{\ast}((\nabla^{\ast}_{Y}A^{\ast})_{V}X).
 \end{eqnarray*}
 \normalsize
\end{corollary}

Next, in the almost product-like statistical manifold $(\widetilde{M}, \widetilde{g},\widetilde{\nabla}, F)$, we put for any vector fields $X,Y,Z$ on $\widetilde{M}$
\begin{eqnarray}
 \widetilde{R}(X,Y)Z&=&\widetilde{c}[\widetilde{g}(Y,Z)X-\widetilde{g}(X,Z)Y+\widetilde{g}(Y,FZ)FX-\widetilde{g}(X,FZ)FY \nonumber\\
  &&+\left\{\widetilde{g}(FX,Y)-\widetilde{g}(X,FY) \right\}FZ], \label{o}
\end{eqnarray}
where $\widetilde{c}$ is a constant. Then the tensor $\widetilde{R}$ is
satisfied the Bianchi's 1st and 2nd identities, and $F\widetilde{R}(X,Y)=\widetilde{R}(X,Y)F$. Moreover, we have
\begin{eqnarray}
 \widetilde{R}^{\ast}(X,Y)Z&=&\widetilde{c}[\widetilde{g}(Y,Z)X-\widetilde{g}(X,Z)Y+\widetilde{g}(Y,F^{\ast}Z)F^{\ast}X-\widetilde{g}(X,F^{\ast}Z)F^{\ast}Y \nonumber \\
  &&+\left\{\widetilde{g}(F^{\ast}X,Y)-\widetilde{g}(X,F^{\ast}Y) \right\}F^{\ast}Z].
\end{eqnarray}
We discuss the statistical submanifold $M$ of the almost product-like statistical manifold $(\widetilde{M}, \widetilde{g},\widetilde{\nabla}, F)$
satisfying the condition (\ref{o}). For any vector fields $X,Y,Z$ tangent to $M$ and $V$ normal to $M$, we get
\begin{eqnarray*}
&&R(X,Y)Z=\widetilde{c}[g(Y,Z)X-g(X,Z)Y+g(Y,fZ)fX-g(X,fZ)fY \\
&&+\left\{\ g(fX,Y)-g(X,fY) \right\}fZ ]+A_{\sigma(Y,Z)}X-A_{\sigma(X,Z)}Y,\\
&&(D_{X}\sigma)(Y,Z)-(D_{Y}\sigma)(X,Z)=\widetilde{c}[g(Y,fZ)hX-g(X,fZ)hY\\
&&+\left\{\ g(fX,Y)-g(X,fY) \right\}hZ], \\
&&(\nabla_{X}A)_{V}Y-(\nabla_{Y}A)_{V}X=\widetilde{c}[-g(Y,tV)fX+g(X,tV)fY-\\
&&\left\{\ g(fX,Y)-g(X,fY) \right\}tV],\\
&&R^{\perp}(X,Y)V\sigma(X,A_{V}Y)+\sigma(Y,A_{V}X)=\widetilde{c}[g(Y,tV)hX-g(X,tV)hY\\
&&+\left\{\ g(fX,Y)-g(X,fY) \right\}sV].
\end{eqnarray*}
\normalsize
If the second fundamental form $\sigma$ is parallel with respect to $D$, then we find from (\ref{o})
\begin{align*}
   \widetilde{c}[g(Y,fZ)hX-g(X,fZ)hY+\left\{\ g(fX,Y)-g(X,fY) \right\}hZ]=0.
\end{align*}
From this equation, we get $\widetilde{c}=0$ or
\begin{align*}
h[g(Y,fZ)X-g(X,fZ)Y+\left\{\ g(fX,Y)-g(X,fY) \right\}Z]=0,
\end{align*}
We assume $g(Y,fZ)X-g(X,fZ)Y+\left\{\ g(fX,Y)-g(X,fY) \right\}Z=0$, which implies that $f=0$ if $ \text{dim}M>2$.

\begin{theorem}
 Let $(M,g)$ be the statistical submanifold of the locally product-like statistical manifold $\widetilde{M}$ satisfying the condition (\ref{o}). If the second fundamental form $\sigma$ is parallel with respect to $D$, then we get
 \begin{itemize}
     \item [(1)] $\widetilde{c}=0$, that is, $\widetilde{M}$ is flat, or
     \item[(2)] $M$ is $F$-invariant of $\widetilde{M}$, or
     \item[(3)] $M$ is $F$-anti-invariant.
 \end{itemize}
\end{theorem}

\begin{corollary}
 Let $(M,g)$ be the statistical submanifold of the locally product-like statistical manifold $\widetilde{M}$ satisfying the condition (\ref{o}). If $M$ is totally geodesic with respect to $\widetilde{\nabla}$, then we get
 \begin{itemize}
     \item [(1)] $\widetilde{c}=0$, that is, $\widetilde{M}$ is flat, or
     \item[(2)] $M$ is $F$-invariant of $\widetilde{M}$, or
      \item[(3)] $M$ is $F$-anti-invariant, which is of constant curvature $\widetilde{c}$.
 \end{itemize}
\end{corollary}
\begin{corollary}
 Let $(M,g)$ be the statistical submanifold of the locally product-like statistical manifold $\widetilde{M}$ satisfying the condition (\ref{o}). If the second fundamental form $\sigma^{\ast}$ is parallel with respect to $D^{\ast}$, then we get
  \begin{itemize}
     \item [(1)] $\widetilde{c}=0$, that is, $\widetilde{M}$ is flat, or
     \item[(2)] $M$ is $F^{\ast}$-invariant of $\widetilde{M}$, or

     \item[(3)] $M$ is $F^{\ast}$-anti-invariant.
 \end{itemize}
\end{corollary}
\begin{corollary}
Let $(M,g)$ be the statistical submanifold of the locally product-like statistical manifold $\widetilde{M}$ satisfying the condition (\ref{o}). If $M$ is totally geodesic with respect to $\widetilde{\nabla}^{\ast}$, then we get \begin{itemize}
     \item [(1)] $\widetilde{c}=0$, that is, $\widetilde{M}$ is flat, or
     \item[(2)] $M$ is $F^{\ast}$-invariant of $\widetilde{M}$, or
     \item[(3)] $M$ is $F^{\ast}$-anti-invariant, which is of constant curvature $\widetilde{c}$.
 \end{itemize}
\end{corollary}

\section{Hypersurfaces of almost product-like Riemannian manifolds}
Let $(\widetilde{M},\widetilde{g},F)$ be an almost product-like Riemannian manifold and $(M,g)$ be a hypersurface of $\widetilde{M}$. If $N$ is the unit normal vector field of $M$, then we put
\begin{eqnarray*}
FN=\xi+\mu_{1}N
\end{eqnarray*}
and
\begin{eqnarray*}
F^{\ast}N=\xi^{\ast}+\mu_{2}N,
\end{eqnarray*}
where $\xi$ and $\xi^{\ast}$ are tangential parts of $FN$ and $F^{\ast}N$, respectively. Here, $\mu_{1}$ and $\mu_{2}$ are smooth functions on $\widetilde{M}$. Considering (\ref{eq6}) in the last two equations, we obtain $\mu_{1}=\mu_{2}$. Thus, we can write
\begin{equation} \label{eq11}
FN=\xi+\mu N
\end{equation}
and
\begin{equation} \label{eq12}
F^{\ast}N=\xi^{\ast}+\mu N,
\end{equation}
where $\mu$ is a smooth function on $\widetilde{M}$. Using (\ref{eq11}) and (\ref{eq12}) in (\ref{eq7}), we derive
\begin{equation} \label{eq13}
1-\mu^{2}=g(\xi,\xi^{\ast}).
\end{equation}
Based on (\ref{eq13}), we obtain the following lemma:
\begin{lemma} \label{lem2}
Let $(M,g)$ be a hypersurface of $\widetilde{M}$. Then we have the following statements:
\begin{itemize}
    \item [i)] If $\mu=1$ then $\xi$ and $\xi^{\ast}$ are orthogonal.
    \item[ii)] If $\mu=0$ then $g(\xi,\xi^{\ast})=1$. \label{itemii}
\end{itemize}
\end{lemma}

\begin{definition}\label{def4}
Let $(M,g)$ be a hypersurface of $\widetilde{M}$. If $FN$ and $F^{\ast}N$ lie on $\Gamma(TM)$, then $(M,g)$ is called a tangential hypersurface.
\end{definition}

Now, let $(M,g)$ be a tangential hypersurface of  $\widetilde{M}$. Then  we find
\begin{equation}\label{eq14}
FN=\xi \quad and \quad F^{\ast}N=\xi^{\ast}.
\end{equation}
From Lemma \ref{lem2},  we note that $\xi$ is not perpendicular to $\xi^{\ast}$  and $(\xi^{\ast})^{\ast}=\xi$.
 For any $X\in \Gamma(TM)$, we write
\begin{equation} \label{eq15}
FX=\varphi X+\eta^{\ast}(X)N
\end{equation}
and
\begin{equation}\label{eq16}
 F^{\ast}X=\varphi^{\ast} X+\eta(X)N,
\end{equation}
where $\varphi X$, $\varphi^{\ast}X \in \Gamma(TM)$, $\eta$ and $\eta^{\ast}$ are $1-$ forms on $M$. Then we get
\begin{equation} \label{etaaa}
\eta^{\ast}(X)=g(X,\xi^{\ast}), \quad \eta(X)=g(X,\xi)
\end{equation}
and $(\eta^{\ast})^{\ast}=\eta$. Moreover, it is easy to see from Lemma \ref{lem2} that $\eta(\xi^{\ast})=1$ and $\eta^{\ast}(\xi)=1$ hold.

\begin{lemma} \label{lemm:3}
For any tangential hypersurface, we have the following relations:
\begin{equation} \label{eq21} \varphi^{2}X=X-\eta^{\ast}(X)\xi,
\end{equation}
\begin{equation}\label{eq22}
\eta^{\ast}(\varphi X)=0,
\end{equation}
\begin{equation} \label{eq23} (\varphi^{\ast})^{2}X=X-\eta(X)\xi^{\ast},
\end{equation}
\begin{equation} \label{eq24}
\eta(\varphi^{\ast}X)=0,
\end{equation}
\begin{equation} \label{eq25}
\varphi \xi=0 \quad and \quad \varphi^{\ast}\xi^{\ast}=0.
\end{equation}
\end{lemma}

\begin{lemma} \label{lemm:2}
   Let $(M,g)$ be tangential hypersurface of $\widetilde{M}$. Then we have the following relations for any $X,Y \in \Gamma(TM)$:
\begin{align*}
 g(\varphi X,Y)&=g(X,\varphi^{\ast}Y), \\
 g(\varphi X,\varphi^{\ast} Y)&=g(X,Y)-\eta^{\ast}(X)\eta(Y).
\end{align*}
Moreover, $(\varphi^{\ast})^{\ast}=\varphi$ holds.
\end{lemma}

\begin{example} \label{exx6}
Let $(\mathbb{R}^{4},\widetilde{g},F)$ be an almost product-like Riemannian manifold of Example \ref{ex1} and $\widetilde{f}$ be an
immersion of Example \ref{exx5}. If we set
\begin{equation*}
 N=e^{-\frac{1}{2}(x_{1}-x_{3}) }\partial_{4},
 \end{equation*}
 then $N$ is an unit normal to $\mathbb{R}^{4}$ and $\xi=e^{-\frac{1}{2}(x_{1}-x_{3}) }\partial_{2}$. We obtain
 \begin{eqnarray*}
 f=\left(
     \begin{array}{cccc}
       0 & 0 & 1 & 0 \\
       0 & 0 & 0 & 0 \\
       1 & 0 & 0 & 0 \\
       0 & 0 & 0 & 0 \\
     \end{array}
   \right), \ h=\left(
                  \begin{array}{cccc}
                    0 & 0 & 0 & 0 \\
                    0 & 0 & 0 & 0 \\
                    0 & 0 & 0 & 0 \\
                    0 & 1 & 0 & 0 \\
                  \end{array}
                \right), \
 \end{eqnarray*}
 \begin{eqnarray*}
t=\left(
    \begin{array}{cccc}
        0 & 0 & 0 & 0 \\
        0 & 0 & 0 & 1 \\
        0 & 0 & 0 & 0 \\
        0 & 0 & 0 & 0 \\
    \end{array}
     \right), \ s=O
 \end{eqnarray*}
 and
 \begin{eqnarray*}
 f^{\ast}=\left(
   \begin{array}{cccc}
     0 & 0 & (1+e^{x_{1}-x_{3}})^{-1} & 0 \\
     0 & 0 & 0 & 0 \\
     1+e^{x_{1}-x_{3}} & 0 & 0 & 0 \\
     0 & 0 & 0 & 0 \\
   \end{array}
 \right), \ \ h^{\ast}=\left(
                         \begin{array}{cccc}
                          0 & 0 & 0 & 0 \\
                           0 & 0 & 0 & 0 \\
                           0 & 0 & 0 & 0 \\
                           0 & 1 & 0 & 0 \\
                         \end{array}
                       \right),
 \end{eqnarray*}
 \begin{eqnarray*}
 t^{\ast}=\left(
            \begin{array}{cccc}
             0 & 0 & 0 & 0  \\
              0 & 0 & 0 & 1  \\
              0 & 0 & 0 & 0  \\
             0 & 0 & 0 & 0  \\
            \end{array}
          \right), \ \ s^{\ast}=O,
 \end{eqnarray*}
 which implies from $s=O$ that the pair $(M,g)$ is a tangential hypersurface. Moreover, we find $\varphi=f$, $\eta^{\ast}(\partial_{1})=0$, $\eta^{\ast}(\partial_{2})=e^{\frac{1}{2}(x_{1}-x_{3}) }$ and $\eta^{\ast}(\partial_{3})=0$.
\end{example}

\section{Tangential hypersurfaces of locally product-like statistical manifolds}
Let $(M,g)$ be a tangential hypersurface of an almost product-like statistical manifold  $\widetilde{M}$. The Gauss and Weingarten formulas
 with respect to $\widetilde{\nabla}$ and $\widetilde{\nabla}^{\ast}$  are given by

\begin{eqnarray*}
\widetilde{\nabla}_{X}Y&=&\nabla_{X}Y+ \ ^{'}\sigma(X,Y)N, \\
\widetilde{\nabla}_{X}N&=&-A_{N}X+\kappa(X)N, \\
 \widetilde{\nabla}^{\ast}_{X}Y&=&\nabla^{\ast}_{X}Y+ \ ^{'}\sigma^{\ast}(X,Y)N,\\
  \widetilde{\nabla}^{\ast}_{X}N&=&-A^{\ast}_{N}X+\kappa^{\ast}(X)N,
\end{eqnarray*}
where $\kappa$ and $\kappa^{\ast}$ are $1-$forms. Then we find $^{'}\sigma(X,Y)=g(A_{N}^{\ast}X,Y)$,
$^{'} \\ \sigma^{\ast}(X,Y)=g(A_{N}X,Y)$, $\kappa(X)+\kappa^{\ast}(X)=0$. Thus, we get
\begin{equation}\label{eq26}
    \widetilde{\nabla}_{X}Y=\nabla_{X}Y+g(A^{\ast}_{N}X,Y)N,
\end{equation}
\begin{equation}\label{eq27}
    \widetilde{\nabla}_{X}N=-A_{N}X+\kappa(X)N,
\end{equation}
\begin{equation}\label{eq28}
    \widetilde{\nabla}^{\ast}_{X}Y=\nabla^{\ast}_{X}Y+g(A_{N} X,Y)N,
\end{equation}
\begin{equation}\label{eq29}
    \widetilde{\nabla}^{\ast}_{X}N=-A^{\ast}_{N}X-\kappa(X)N.
\end{equation}

\begin{proposition} \label{prop:5a}
Let $(M,g)$ be a tangential hypersurface of a locally product-like  statistical manifold $\widetilde{M}$. Then we have the following relations:
\begin{equation}\label{eq32}
\nabla_{X} \xi=-\varphi (A_{N}X)+\kappa(X)\xi,
\end{equation}
\begin{equation}\label{eq34}
\nabla^{\ast}_{X} \xi^{\ast}=-\varphi^{\ast} (A^{\ast}_{N}X)-\kappa(X)\xi^{\ast},
\end{equation}
\begin{equation} \label{eq33ab}
\eta\left( A_{N}^{\ast}X\right)+\eta^{\ast}\left( A_{N}X\right)=0, \quad \quad A_{N}\xi^{\ast}+A_{N}^{\ast}\xi=0,
\end{equation}
\begin{equation}\label{eq33}
\kappa(X)=\eta^{\ast}(\nabla_{X} \xi)=-\eta(\nabla^{\ast}_{X}\xi^{\ast}).
\end{equation}
\end{proposition}
\begin{proof}
Using $FN=\xi$, $F^{\ast}N=\xi^{\ast}$,  (\ref{eq15}) and (\ref{eq28}), we can write
\begin{eqnarray*}
  (\widetilde{\nabla}_{X}F)N&=&\nabla_{X}\xi+\varphi(A_{N}X)-\kappa(X)\xi+\{\eta(A_{N}^{\ast}X)+\eta^{\ast}(A_{N}X)\}N,\\
   (\widetilde{\nabla}^{\ast}_{X}F^{\ast})N&=& \nabla^{\ast}_{X}\xi^{\ast}+\varphi^{\ast}(A^{\ast}_{N}X)-\kappa(X)\xi^{\ast}+\{\eta(A_{N}^{\ast}X)+\eta^{\ast}(A_{N}X)\}N,\\
\end{eqnarray*}
which implies from $\widetilde{\nabla}F=\widetilde{\nabla}^{\ast}F^{\ast}=0$ that  the proof is completed.
\end{proof}
\begin{proposition} \label{prop:8a}
Let $(M,g)$ be a  tangential hypersurface of a locally product-like statistical manifold $\widetilde{M}$.
Then we have
\begin{equation}\label{eq40}
  (\nabla_{X}\varphi)Y=g(A^{\ast}_{N}X,Y)\xi+\eta^{\ast}(Y)A_{N}X,
\end{equation}
\begin{equation}\label{eq42}
(\nabla^{\ast}_{X}\varphi^{\ast})Y=g(A_{N}X,Y)\xi^{\ast}+\eta(Y)A^{\ast}_{N}X.
\end{equation}
\end{proposition}
\begin{proof}
Using (\ref{eq26}) and (\ref{eq28}), we get
\begin{eqnarray*}
(\widetilde{\nabla}_{X}F)Y=(\nabla_{X}\varphi)Y-g(A_{N}^{\ast}X,Y)\xi-\eta^{\ast}(Y)A_{N}X,
\end{eqnarray*}
\begin{eqnarray*}
(\widetilde{\nabla}^{\ast}_{X}F^{\ast})Y=(\nabla^{\ast}_{X}\varphi^{\ast})Y-g(A_{N}X,Y)\xi^{\ast}-\eta(Y)A^{\ast}_{N}X,
\end{eqnarray*}
which yields from $\widetilde{\nabla}F=\widetilde{\nabla}^{\ast}F^{\ast}=0$ that $(\ref{eq40})$ and $(\ref{eq42})$ hold.
\end{proof}

\begin{example}
Let $(\mathbb{R}^{4},\widetilde{g},F)$ be a locally product-like statistical manifold of Example \ref{exx2}
and $(M,g)$ be a tangential hypersurface of Example \ref{exx6}. By straightforward calculate we find
\begin{eqnarray*}
\nabla_{\partial_{1}}\xi=\nabla_{\partial_{3}}\xi=-\frac{1}{2}(1+2e^{-x_{1}+x_{3}})\xi, \ \
\nabla_{\partial_{2}}\xi=-e^{\frac{1}{2}(x_{1}-x_{3})}(\partial_{2}+\partial_{3})+\xi
\end{eqnarray*}
and
\begin{align*}
(\nabla_{\partial_{1}}\varphi)\partial_{1}&=0, &  (\nabla_{\partial_{2}}\varphi)\partial_{1}&=-2\cosh\frac{1}{2}(x_{1}-x_{3})\xi, &  (\nabla_{\partial_{3}}\varphi)\partial_{1}&=0, \\
(\nabla_{\partial_{1}}\varphi)\partial_{2}&=0, &  (\nabla_{\partial_{2}}\varphi)\partial_{2}&=e^{x_{1}-x_{3}}(\partial_{1}+\partial_{3}), &  (\nabla_{\partial_{3}}\varphi)\partial_{2}&=0, \\
(\nabla_{\partial_{1}}\varphi)\partial_{3}&=0, &  (\nabla_{\partial_{2}}\varphi)\partial_{3}&=-e^{-\frac{1}{2}(x_{1}-x_{3})}\xi, &  (\nabla_{\partial_{3}}\varphi)\partial_{3}&=0.
\end{align*}
Also, we get
\begin{align*}
A_{N}\partial_{1}&=-A^{\ast}_{N}\partial_{1}=(1+e^{-x_{1}+x_{3}})\xi,  & \kappa(\partial_{1})&=-\frac{1}{2}(1+2e^{-x_{1}+x_{3}}),\\
A_{N}\partial_{2}&=-A^{\ast}_{N}\partial_{2}=e^{\frac{1}{2}(x_{1}-x_{3})}(\partial_{1}+\partial_{3})+\xi,  & \kappa(\partial_{2})&=1, \\
A_{N}\partial_{3}&=-A^{\ast}_{N}\partial_{3}=e^{-x_{1}+x_{3}}\xi,  & \kappa(\partial_{3})&=-\frac{1}{2}(1+2e^{x_{1}-x_{3}}),
\end{align*}
which yields that Proposition \ref{prop:5a} and Proposition \ref{prop:8a} hold.
\end{example}

From Proposition \ref{Takano:prop:1e} and Proposition \ref{prop:5a}, we have

\begin{proposition}
Let $(M,g)$ be a tangential hypersurface of a locally product-like statistical manifold $\widetilde{M}$.
Then we obtain
\begin{align*}
R(X,Y)\xi&=-\varphi \left ( (\overline{\nabla}_{X}A)_{N}Y \right )+\varphi \left ( (\overline{\nabla}_{Y}A)_{N}X \right )-\eta^{\ast}(A_{N}Y)A_{N}X\\
&\quad+\eta^{\ast}(A_{N}X)A_{N}Y-\left\{ g([A_{N},A^{\ast}_{N}]X,Y)-(d\kappa)(X,Y) \right\}\xi,\\
R^{\ast}(X,Y)\xi^{\ast}&=-\varphi^{\ast} \left ( (\overline{\nabla}^{\ast}_{X}A^{\ast})_{N}Y \right )+\varphi^{\ast} \left ( (\overline{\nabla}^{\ast}_{Y}A^{\ast})_{N}X \right )-\eta(A^{\ast}_{N}Y)A^{\ast}_{N}X\\
&\quad+\eta(A^{\ast}_{N}X)A^{\ast}_{N}Y+\left\{ g([A_{N},A^{\ast}_{N}]X,Y)-(d\kappa)(X,Y) \right\}\xi^{\ast},\\
\end{align*}
where we put
\begin{align*}
 (\overline{\nabla}_{X}A)_{N}Y&=\nabla_{X}(A_{N}Y)-A_{\widetilde{\nabla}_{X}N}Y-A_{N}(\nabla_{X}Y), \\
 (\nabla_{X}\kappa)(Y)&=X\left\{ \kappa(Y) \right\}-\kappa (\nabla_{X}Y).
\end{align*}
\end{proposition}

\begin{proposition}
Let $(M,g)$ be a tangential hypersurface of a locally product-like statistical manifold $\widetilde{M}$.
Then we obtain
\begin{align*}
\widetilde{R}(X,Y)Z&=R(X,Y)Z-g(A^{\ast}_{N}Y,Z)A_{N}X+g(A^{\ast}_{N}X,Z)A_{N}Y\\
 &\quad +\{g( (\overline{\nabla}^{\ast}_{X}A^{\ast})_{N}Y,Z)-g( (\overline{\nabla}^{\ast}_{Y}A^{\ast})_{N}X,Z)\}N,\\
 \widetilde{R}(X,Y)N&=-(\overline{\nabla}_{X}A)_{N}Y+(\overline{\nabla}_{Y}A)_{N}X -\left\{ g([A_{N},A^{\ast}_{N}]X,Y)\right.\\
 &\quad \left.-(d\kappa)(X,Y) \right\}N,\\
 \widetilde{R}^{\ast}(X,Y)Z&=R^{\ast}(X,Y)Z-g(A_{N}Y,Z)A^{\ast}_{N}X+g(A_{N}X,Z)A^{\ast}_{N}Y\\
 &\quad +\{g( (\overline{\nabla}_{X}A)_{N}Y,Z)-g( (\overline{\nabla}_{Y}A)_{N}X,Z)\}N,\\
 \widetilde{R}^{\ast}(X,Y)N&=-(\overline{\nabla}^{\ast}_{X}A^{\ast})_{N}Y+(\overline{\nabla}^{\ast}_{Y}A^{\ast})_{N}X\\
 &\quad +\left\{ g([A_{N},A^{\ast}_{N}]X,Y)-(d\kappa)(X,Y) \right\}N.
\end{align*}
\end{proposition}

From $\widetilde{R}(X,Y)FZ=F(\widetilde{R}(X,Y)Z)$ and $\widetilde{R}^{\ast}(X,Y)F^{\ast}Z=F^{\ast}(\widetilde{R}^{\ast}(X,Y)Z)$, we get

\begin{proposition} \label{prop.new:Takano}
Let $(M,g)$ be a tangential hypersurface of a locally product-like statistical manifold $\widetilde{M}$.
Then we obtain
\begin{align*}
&R(X,Y)\varphi Z+g(A^{\ast}_{N}X,\varphi Z)A_{N}Y-g(A^{\ast}_{N}Y,\varphi Z)A_{N}X\\
&-\eta^{\ast}(Z) \left\{ (\overline{\nabla}_{X}A)_{N}Y-(\overline{\nabla}_{Y}A)_{N}X \right\}=\varphi (R(X,Y)Z)-g(A^{\ast}_{N}Y,Z)\varphi (A_{N}X)\\
&+g(A^{\ast}_{N}X,Z)\varphi (A_{N}Y)+\{g((\overline{\nabla}^{\ast}_{X}A^{\ast})_{N}Y,Z)-g((\overline{\nabla}^{\ast}_{Y}A^{\ast})_{N}X,Z) \}\xi,
\end{align*}
and
\begin{align*}
&R^{\ast}(X,Y)\varphi^{\ast}Z+g(A_{N}X,\varphi^{\ast} Z)A^{\ast}_{N}Y-g(A_{N}Y,\varphi^{\ast} Z)A^{\ast}_{N}X\\
&-\eta(Z) \left\{ (\overline{\nabla}^{\ast}_{X}A^{\ast})_{N}Y-(\overline{\nabla}^{\ast}_{Y}A^{\ast})_{N}X \right\}=\varphi^{\ast} (R^{\ast}(X,Y)Z)+\left\{g((\overline{\nabla}_{X}A)_{N}Y,Z)\right.\\
&\left.-g((\overline{\nabla}_{Y}A)_{N}X,Z) \right\}\xi^{\ast}-g(A_{N}Y,Z)\varphi^{\ast}(A^{\ast}_{N}X)+g(A_{N}X,Z)\varphi^{\ast} (A^{\ast}_{N}Y).
\end{align*}
\end{proposition}

From $\widetilde{R}(X,Y)FN=F(\widetilde{R}(X,Y)N)$,  $\widetilde{R}^{\ast}(X,Y)F^{\ast}N=F^{\ast}(\widetilde{R}^{\ast}(X,Y)N)$, we get
\begin{proposition}
Let $(M,g)$ be a tangential hypersurface of a locally product-like statistical manifold $\widetilde{M}$.
Then we obtain
\begin{eqnarray*}
   \eta((\overline{\nabla}^{\ast}_{X}A^{\ast})_{N}Y)-\eta ((\overline{\nabla}^{\ast}_{Y}A^{\ast})_{N}X)=-\eta^{\ast}((\overline{\nabla}_{X}A)_{N}Y)+\eta^{\ast}((\overline{\nabla}_{Y}A)_{N}X).
\end{eqnarray*}
\end{proposition}

Next, we discuss the tangential hypersurface $(M,g)$ of an almost product-like statistical manifold
satisfying (\ref{o}). Then we find from Proposition \ref{prop.new:Takano}
\begin{eqnarray}
R(X,Y)\varphi Z &=& \widetilde{c}\left[ g(Y,Z)X-g(X,Z)Y+g(Y,\varphi Z)\varphi X-g(X,\varphi Z)\varphi Y \right. \nonumber \\
&&\left.+\{g(\varphi X,Y)-g(X,\varphi Y) \}\varphi Z\right]+g(A^{\ast}_{N}Y,Z)A_{N}X \nonumber \\
&&-g(A^{\ast}_{N}X,Z)A_{N}Y, \label{Tk:1}
\end{eqnarray}
\begin{align}
g((\overline{\nabla}^{\ast}_{X}A^{\ast})_{N}Y,Z)-g((\overline{\nabla}^{\ast}_{Y}A^{\ast})_{N}X,Z)&=\widetilde{c}[\eta^{\ast}(X)g(Y,\varphi Z)-\eta^{\ast}(Y)g(X,\varphi Z) \nonumber \\
&\quad+\eta^{\ast}(Z)\{g(\varphi X,Y)-g(X,\varphi Y)\}], \label{Tk:2}
\end{align}
\begin{align}
(\overline{\nabla}_{X}A)_{N}Y-(\overline{\nabla}_{Y}A)_{N}X&=\widetilde{c}\left[\eta(X)\varphi Y-\eta(Y)\varphi X\right. \nonumber \\
&\quad\left. +\{g(X,\varphi Y)-g(\varphi X,Y)\}\xi\right], \label{Tk:3}
\end{align}
\begin{eqnarray}
g([A_{N},A_{N}^{\ast}]X,Y)-(d \kappa)(X,Y)=\widetilde{c}\{\eta(X)\eta^{\ast}(Y)-\eta(Y)\eta^{\ast}(X)\}. \label{Tk:4}
\end{eqnarray}
If $M$ is totally geodesic with respect to $\widetilde{\nabla}$, that is, $'\sigma=0$, then we get from  (\ref{Tk:3})
\begin{eqnarray*}
 \widetilde{c}[\eta^{\ast}(X)g(Y,\varphi Z)-\eta^{\ast}(Y)g(X,\varphi Z)+\eta^{\ast}(Z)\{g(\varphi X,Y)-g(X,\varphi Y)\}]=0,
\end{eqnarray*}
which denotes that $\widetilde{c}=0$  or
\begin{equation}
\eta^{\ast}(X)g(Y,\varphi Z)-\eta^{\ast}(Y)g(X,\varphi Z)+\eta^{\ast}(Z)\{g(\varphi X,Y)-g(X,\varphi Y)\}=0.\label{Tk:5}
\end{equation}
Therefore we get $g(\varphi X,Y)=g(X,\varphi Y)$, namely, $\varphi^{\ast}=\varphi$. It is east to see from (\ref{Tk:5}) that
$\eta^{\ast}(X)\varphi Z=g(X,\varphi Z)\xi^{\ast}$ holds. Thus $X=\eta^{\ast}(X)\xi=\eta(X)\xi^{\ast}$ holds. Because of any vector
field $X$ tangent to $M$ is parallel to $\xi$ and $\xi^{\ast}$, we can not select to the $n$th linearly independence vector
fields, where $n=\textrm{dim}M(>2)$. Hence we have
\begin{theorem}
Let $(M,g)$ be the tangential hypersurface of the locally product-like statistical manifold $\widetilde{M}$
satisfying the condition (\ref{o}). If $M$ is totally geodesic with respect to $\widetilde{\nabla}$, then
$\widetilde{M}$ is flat.
\end{theorem}

\section{Conclusions and future works}

The theory of hypersurfaces of Riemannian manifolds admitting various differentiable structures includes comprehensive geometric properties.
For example, every hypersurface of an almost complex space form possesses a contact structure (cf. \cite{Adachi,Chen:2,Kon}), every hypersurface of an almost contact space form possesses a natural $f$-structure (cf.\cite{Blair,Eum}) and every hypersurface of almost product Riemannian manifolds admits a para contact structure under some conditions (cf. \cite{Adati,Deshmukh}). With the help of these features,
each hypersurface of an almost complex space form becomes a contact metric manifold, each hypersurface of an almost contact space form becomes a metric $f$-manifold and some special hypersurfaces become a para contact manifold. Thus, the basic relations and properties of contact metric manifolds,
metric $f$-manifolds, and para contact metric manifolds can be constructed by examining the basic properties of the geometry of these hypersurfaces. In this study, as a result of examining tangential hypersurfaces of almost product-like Riemannian manifolds, para contact-like structures have coincided.

\medskip

Para contact manifolds firstly defined by K. Sato and K. K. Matsumoto in \cite{Sato} as follows:

\medskip

Let $(\overline{M},\overline{g})$ be a $m$-dimensional Riemannian manifold, $\overline{\varphi}$ be a $(1,1)$ tensor field,
$\overline{\xi}$ be a tangent vector field and $\overline{\eta}$ be a $1$-form on $M$. Then $(\overline{M},\overline{g},\overline{\varphi},\overline{\xi},\overline{\eta})$
is called an almost para contact Riemannian manifold if the following relations are satisfied for any $X,Y\in \Gamma(T\overline{M})$:
\begin{eqnarray*}
 \overline{\varphi}^{2}X=X-\overline{\eta}(X)\overline{\xi}, \ \ \ \overline{\eta}(\overline{\xi})=1, \ \ \ \overline{\varphi} \overline{\xi}=0,  \ \ \ \overline{\eta}(\overline{\varphi} X)=0
\end{eqnarray*}
and
\begin{eqnarray*}
 g(\overline{\varphi} X,\overline{\varphi} Y)=\overline{g}(X,Y)-\overline{\eta}(X)\overline{\eta}(Y).
\end{eqnarray*}

\medskip

Considering the definition of almost para contact manifolds,   it is possible to give the following definitions:

\begin{definition} \label{Defn:para1}
 A Riemannian manifold $(M,g)$  is called an almost para contact-like manifold if there exists a differentiable structure
$(\varphi,\varphi^{\ast},\xi,\xi^{\ast},\eta,\eta^{\ast})$
consisting of tensor fields $\varphi$ and $\varphi^{\ast}$ of type $(1,1)$,
vector fields $\xi$ and $\xi^{\ast}$, $1$-forms $\eta$ and $\eta^{\ast}$ such that the following relations are satisfied
\begin{eqnarray*}
&&\varphi ^{2}=I-\eta^{\ast} \otimes \xi, \ \ \ \varphi \xi=0, \ \ \ \eta(\xi^{\ast})= 1, \ \ \  \eta\circ \varphi^{\ast}=0, \\
&& (\varphi^{\ast}) ^{2}=I-\eta \otimes \xi^{\ast}, \ \ \ \varphi^{\ast} \xi^{\ast}=0, \ \ \ \eta^{\ast}(\xi)= 1
 \ \ \  \eta^{\ast}\circ \varphi=0,
\end{eqnarray*}
where $I$ denotes the identity map.
\end{definition}

\begin{definition} \label{Defn:para2}
An almost para contact-like manifold $(M,g,\varphi,\varphi^{\ast},\xi,\xi^{\ast},\eta,\eta^{\ast})$ is called a para contact-like metric manifold if
the following relation is satisfied for any $X,Y\in \Gamma(TM)$:
\begin{eqnarray*}
 g(\varphi X,\varphi^{\ast} Y)=g(X,Y)-\eta^{\ast}(X)\eta(Y).
\end{eqnarray*}
\end{definition}

From Lemma \ref{lemm:3}, Lemma \ref{lemm:2}, Definition \ref{Defn:para1} and Definition \ref{Defn:para2}, we get
\begin{corollary}
Every tangential hypersurface of an almost product-like manifold is a para contact-like metric manifold.
\end{corollary}

Considering Definition \ref{Defn:para1} and Definition \ref{Defn:para2},
the problem of examining the geometric and physical properties naturally arises in para contact-like manifolds and their submanifolds.



Esra Erkan \newline
Address: Department of Mathematics, Faculty of Science and Art, Harran University, \newline
Sanliurfa, TURKEY. \newline
e-mail: esraerkan@harran.edu.tr

\medskip

Kazuhiko Takano \newline
Address: Department of Mathematics,  School of General Education, Shinshu University, Nagano 390-8621, Japan. \newline
e-mail: ktakano@shinshu-u.ac.jp

\medskip

Mehmet G\"{u}lbahar \newline
Address: Department of Mathematics,  Faculty of Science and Art, Harran University, \newline
Sanliurfa, TURKEY. \newline
email:mehmetgulbahar@harran.edu.tr

\end{document}